\documentclass[12pt]{amsart}

\newtheorem{theorem}{Theorem}[section]
\newtheorem{proposition}[theorem]{Proposition}
\newtheorem{lemma}[theorem]{Lemma}
\newtheorem{corollary}[theorem]{Corollary}
\newtheorem{remark}[theorem]{Remark}

\numberwithin{equation}{section}

\begin{document}

\title[Yang--Mills connections on compact complex tori]{Yang--Mills
connections on compact complex tori}

\author{Indranil Biswas}

\address{School of Mathematics, Tata Institute of Fundamental
Research, Homi Bhabha Road, Bombay 400005, India}

\email{indranil@math.tifr.res.in}

\subjclass[2000]{53C07, 32L05, 14L10}

\keywords{Higgs $G$-bundle, Yang-Mills connection, Einstein-Hermitian connection,
polystability, complex torus}

\date{}

\begin{abstract}
Let $G$ be a connected reductive complex affine algebraic group and $K\,\subset\, G$ a
maximal compact subgroup. Let $M$ be a compact complex torus equipped with a flat K\"ahler
structure and $(E_G\, ,\theta)$ a polystable Higgs $G$--bundle on $M$. Take any $C^\infty$
reduction of structure group $E_K\, \subset\, E_G$ to the subgroup $K$ that solves the
Yang--Mills equation for $(E_G\, ,\theta)$. We prove that the principal $G$--bundle $E_G$
is polystable and the above reduction $E_K$ solves the Einstein--Hermitian equation for
$E_G$. We also prove that for a semistable (respectively, polystable) Higgs $G$--bundle
$(E_G \, , \theta)$ on a compact connected Calabi--Yau manifold, the underlying
principal $G$-- bundle $E_G$ is semistable (respectively, polystable).

\end{abstract}

\maketitle

\section{Introduction}\label{sec1}

Let $X$ be a compact connected K\"ahler manifold equipped with a K\"ahler form
$\widetilde\omega$. Let $(E\, ,\theta)$ be a Higgs vector bundle on $X$. Given a
Hermitian structure $h$ on $E$, the curvature of the corresponding Chern connection
on $E$ will be denoted by ${\mathcal K}_h$. A Hermitian structure $h$ is said to satisfy
the Yang--Mills equation for $(E\, ,\theta)$ if there is $c\, \in\, \mathbb R$ such that
$$
\Lambda_{\widetilde\omega}({\mathcal K}_h +\theta\wedge\theta^*)\,=\, c
\sqrt{-1}\cdot\text{Id}_E\, ,
$$
where $\Lambda_{\widetilde\omega}$ is the adjoint of multiplication of forms by
$\widetilde\omega$, and $\theta^*$ is the adjoint of $\theta$ with respect to $h$.
A Higgs bundle admits a Hermitian structure satisfying the Yang--Mills equation if and only
if it is polystable \cite{Si1}, \cite{Hi}.
If $\theta\,=\, 0$, then the above Yang--Mills equation is also
known as the Einstein--Hermitian equation. A holomorphic vector bundle $E$ admits an
Einstein--Hermitian metric if and only if $E$ is polystable \cite{UY}, \cite{Do}.

More generally, let $G$ be a connected reductive affine algebraic group over $\mathbb C$.
Fix a maximal compact subgroup $K\, \subset\, G$. The center of the Lie algebra of $K$ will
be denoted by ${\mathfrak z}({\mathfrak k})$. Let $(E_G\, ,\theta)$ be a Higgs $G$--bundle
on $X$ (its definition is recalled in Section \ref{se4.1}). A $C^\infty$ reduction of
structure group of $E_G$ to $K$
$$
E_K\, \subset\, E_G
$$
is said to satisfy the Yang--Mills equation for $(E\, ,\theta)$ if there is an element
$c\, \in\, {\mathfrak z}({\mathfrak k})$ such that
$$
\Lambda_{\widetilde\omega}({\mathcal K} +\theta\wedge\theta^*)\,=\, c\, ,
$$
where ${\mathcal K}$ is the curvature of the Chern connection associated to
the reduction $E_K$ and $\theta^*$ is the adjoint of $\theta$
constructed using $E_K$ (see \cite[p. 191--192, Proposition
5]{At} for Chern connections on principal bundles). A Higgs $G$--bundle admits a
Yang--Mills connection if and only if it is polystable \cite{Si2}, \cite{BS}.
As mentioned before, if $\theta\,=\, 0$, then the Yang--Mills equation is also called
the Einstein--Hermitian equation.

We consider Higgs $G$--bundles over a torus $M$ equipped with a flat K\"ahler form
$\widetilde\omega$. If $(E_G\, ,\theta)$ is a polystable Higgs $G$--bundle on $M$, we
prove that the principal $G$--bundle $E_G$ is polystable. If a reduction to $K$
$$
E_K\, \subset\, E_G
$$
satisfies the Yang--Mills equation for $(E\, ,\theta)$, we show that
$E_K$ also satisfies the Einstein--Hermitian equation for $E_G$.

In the last section we observe some properties of Higgs bundles on K\"ahler manifolds
with nonnegative tangent bundle.

\section{Higgs vector bundles on a torus}

\subsection{Semistable and polystable Higgs bundles}

Let $M$ be a compact complex torus of complex dimension $d$. Fix a K\"ahler
class
$$
\omega\,\in\, H^{1,1}(M)\cap H^2(M,\, {\mathbb R})\, .
$$
The degree of any torsionfree coherent analytic sheaf $F$ on $M$ is defined
to be
$$
\text{degree}(F)\, :=\, (c_1(F)\cap \omega^{d-1})\cap [M]\, \in\, {\mathbb R}\, .
$$
If $F$ is of positive rank, then
$$
\mu(F)\, :=\, \frac{{\rm degree}(F)}{{\rm rank}(F)}\, \in\, {\mathbb R}
$$
is called the \textit{slope} of $F$.

Let $E$ be a holomorphic vector bundle on $M$. A {\it Higgs field} on $E$ is a
holomorphic section
$$
\theta\, \in\, H^0(M,\, End(E)\otimes\Omega^1_M)
$$
such that section
$$
\theta\wedge\theta\, \in\, H^0(M,\, End(E)\otimes\Omega^2_M)
$$
vanishes identically. A \textit{Higgs bundle} is a holomorphic vector bundle equipped
with a Higgs field.

The following lemma is well-known (see \cite{BF}, \cite{FGN}).

\begin{lemma}\label{lem1}
Let $(E\, ,\theta)$ be a semistable Higgs bundle on $M$. Then the holomorphic
vector bundle $E$ is semistable.
\end{lemma}

\begin{proof}
Assume that $E$ is not semistable. Let
\begin{equation}\label{ee1}
E_1\,\subset\, E_2\,\subset\, \cdots \,\subset\, E_n\,=\, E
\end{equation}
be the Harder--Narasimhan filtration of $E$. We have
$$
H^0(M,\, End(E_1))\, =\, H^0(M,\, Hom(E_1\, ,E))
$$
because $H^0(M,\, Hom(E_1\, ,E_i/E_{i-1}))\,=\, 0$ for every
$i\, \in\, \{2\, ,\cdots\, , n\}$. Since $\Omega^1_M$ is the trivial vector
bundle of rank $d$, this implies that
$$
\theta(E_1)\, \subset\, E_1\otimes \Omega^1_M\, .
$$
Therefore, $E_1$ contradicts the given condition that $(E\, ,\theta)$ is
semistable. Hence the vector bundle $E$ is semistable.
\end{proof}

We note that the following lemma is a consequence of Corollary 2.2 of \cite[p. 73]{Bi}
(see also \cite{FGN}).

\begin{lemma}\label{lem2}
Let $(E\, ,\theta)$ be a polystable Higgs bundle on $M$. Then the holomorphic
vector bundle $E$ is polystable. 
\end{lemma}

\begin{proof}
{}From Lemma \ref{lem1} we know that $E$ is semistable. Let
$$
V\,\subset\, E
$$
be the coherent analytic subsheaf generated by all polystable subsheaves of $E$
with slope $\mu(E)$. This $V$ is a polystable subsheaf with slope $\mu(E)$
\cite[page 23, Lemma 1.5.5]{HL}. We will show that
\begin{equation}\label{s1}
\theta(V)\, \subset\, V\otimes\Omega^1_M\, .
\end{equation}

To show \eqref{s1}, fix a holomorphic trivialization of $\Omega^1_M$. Using this
trivialization, the homomorphism $\theta$ is written as
$$
\theta\,=\, (\theta_1\, ,\cdots\, ,\theta_d)\, ,
$$
where $\theta_i\, \in\, H^0(M,\, End(E))$ for every $i$. Since $E$ is semistable, and $V$
is polystable with $\mu(V)\,=\, \mu(E)$, it follows that $\theta_i(V)$ is polystable with
$\mu(\theta_i(V))\,=\, \mu(E)$. Therefore, \eqref{s1} holds.

Assume that $E$ is not polystable. So $V\,\not=\, E$.
Since the Higgs bundle $(E\, ,\theta)$ is polystable, from \eqref{s1} and the fact
that $\mu(V)\,=\, \mu(E)$ we conclude that there is a coherent analytic subsheaf
$$
V'\, \subset\, E
$$
such that $\mu(V')\,=\, \mu(E)$ and $V\cap V'\,=\, 0$. Let $V''\,\subset\, V'$ be
a polystable subsheaf such that $\mu(V'')\,=\,\mu(V')$. From the definition of $V$
it follows that $V''\,\subset\, V$. But this contradicts the condition that
$V\cap V''\, \subset\, V\cap V'\,=\, 0$. So $E$ is polystable.
\end{proof}

\subsection{Higgs fields on a polystable vector bundle}

Let $E\,\longrightarrow\, M$ be a polystable vector bundle. Our aim in this subsection
is to describe all Higgs fields $\theta$ on $E$ such that the Higgs
bundle $(E\, ,\theta)$ is polystable.

Since $E$ is polystable, we can write
\begin{equation}\label{e1}
E\,=\, \bigoplus_{j=1}^\ell E_j\otimes {\mathbb C}^{n_j}\, ,
\end{equation}
where
\begin{itemize}
\item each $E_j$ is a stable vector bundle with $\mu(E_j)\,=\, \mu(E)$,

\item $E_j$ is not isomorphic to $E_{j'}$ if $j\,\not=\, j'$, and

\item $n_j\, >\, 0$ for every $j$.
\end{itemize}
{}From the first two conditions if follows immediately that
$H^0(M,\, Hom(E_j\, ,E_{j'}))\,=\, 0$ if $j\,\not=\, j'$. Therefore, we have
\begin{equation}\label{e2}
H^0(M,\, Hom(E_j\, ,E_{j'})\otimes\Omega^1_M)\,=\, 0~\ \text{ if }~ j\,\not=\, j'\, .
\end{equation}
Since $E_j$ is stable, we also have
\begin{equation}\label{e3}
H^0(M,\, End(E_j))\,=\, \mathbb C\, .
\end{equation}
In view of \eqref{e2} and \eqref{e3}, any $\beta\,\in\,
H^0(M,\, End(E))$ can be written as
$$
\beta\,=\, \bigoplus_{j=1}^\ell \text{Id}_{E_j}\otimes T_j
$$
in terms of the isomorphism in \eqref{e1}, where
$$
T_j\, \in\, M(n_j, {\mathbb C})\,=\, {\rm End}_{\mathbb C}({\mathbb C}^{n_j})\, .
$$

As before, fix a holomorphic trivialization of $\Omega^1_M$. Using this trivialization,
any $\theta\,\in\, H^0(M,\, End(E)\otimes\Omega^1_M)$ can be written as
$$
\theta\,=\, (\theta_1\, ,\cdots\, ,\theta_d)\, ,
$$
where $\theta_i\,\in\, H^0(M,\, End(E))$.

Take any
\begin{equation}\label{e4}
\theta\,\in\, H^0(M,\, End(E)\otimes\Omega^1_M)\, .
\end{equation}
Write
$$
\theta\,=\, (\theta_1\, ,\cdots\, ,\theta_d)\, ,
$$
as above. Let
\begin{equation}\label{e5}
\theta_i\, =\, \bigoplus_{j=1}^\ell \text{Id}_{E_j}\otimes T^i_j\, ,
\end{equation}
where $T^i_j\, \in\, M(n_j, {\mathbb C})$.

\begin{proposition}\label{prop1}
The pair $(E\, ,\theta)$ in \eqref{e4} is a polystable Higgs bundle if and only if 
\begin{enumerate}
\item $T^i_jT^k_j\,=\, T^k_jT^i_j$ (see \eqref{e5}) for all $i\, ,k\,\in\, \{1\, ,\cdots\,
,d\}$ and all $j$, and

\item each $T^i_j$ is semisimple.
\end{enumerate}
\end{proposition}

\begin{proof}
First assume that the two conditions in the proposition are satisfied. The first condition
implies that $\theta \wedge\theta\, =\, 0$. The second condition implies that
$(E\, ,\theta)$ can be expressed as a direct sum of stable Higgs bundles of same slope.
Therefore, $(E\, ,\theta)$ is polystable.

Now assume that $(E\, ,\theta)$ is a polystable Higgs bundle. Since $\theta \wedge\theta\, =
\, 0$, the first condition in the proposition holds. The Higgs bundle $(E\, ,\theta)$ is a
direct sum of stable Higgs bundles of same slope. From this it follows that the second
condition in the proposition is satisfied.
\end{proof}

\begin{remark}\label{rem1}
{\rm A sum of commuting semisimple matrices is again semisimple. Therefore, the two
conditions in Proposition \ref{prop1} are independent of the choice of the trivialization
of $\Omega^1_M$.}
\end{remark}

\section{Yang--Mills Hermitian metric on polystable Higgs bundles}\label{sec3}

Let $\text{Aut}^0(M)$ denote the connected component, containing the identity element,
of the group of holomorphic automorphisms of $M$. The complex manifold $\text{Aut}^0(M)$
is isomorphic to $M$. If we consider $M$ as a complex abelian Lie group, then
$\text{Aut}^0(M)$ coincides with the group of translations of $M$.

There is a unique K\"ahler form $\widetilde\omega$ on $M$ such that
\begin{itemize}
\item the cohomology class of $\widetilde\omega$ coincides with $\omega$, and

\item the form $\widetilde\omega$ is preserved by the action of $\text{Aut}^0(M)$ on $M$.
\end{itemize}
The K\"ahler structure on $M$ given by $\widetilde\omega$ is flat. Fix the
K\"ahler form $\widetilde\omega$ on $M$.

\begin{proposition}\label{pr1}
Let $(E\, ,\theta)$ be a polystable Higgs bundle on $M$. There is a
Yang--Mills Hermitian metric $h$ on $E$ for the Higgs field $\theta$ such that
$h$ satisfies the Einstein--Hermitian equation for the polystable vector bundle $E$.
\end{proposition}

\begin{proof}
Fix a trivialization of $\Omega^1_M$ using holomorphic sections of $\Omega^1_M$ 
that are pointwise orthonormal. Such a trivialization exists because
the connection on $\Omega^1_M$ corresponding to $\widetilde\omega$ is flat with
trivial monodromy. Take the endomorphisms $T^i_j$ in Proposition \ref{prop1}. 
Take any $j\, \in\, \{1\, ,\cdots\, , \ell\}$. Since $T^i_jT^k_j\,=\, T^k_jT^i_j$
for all $i\, ,k\,\in\, \{1\, ,\cdots\, ,d\}$, we have a
simultaneous eigenspace decomposition of
${\mathbb C}^{n_j}$ for the eigenvalues of $T^i_j$, $i\,\in\, \{1\, ,\cdots\, ,d\}$.
Fix an inner product $h_j$ on ${\mathbb C}^{n_j}$ such that the above decomposition
of ${\mathbb C}^{n_j}$ given by the eigenspaces of $\{T^i_j\}_{i=1}^d$ is orthogonal.

Fix an Einstein--Hermitian structure $h'_j$ on the stable vector bundle $E_j$ in
\eqref{e1}. The Hermitian structures $h_j$ and $h'_j$ together produce
a Hermitian structure on the vector bundle $E_j\otimes {\mathbb C}^{n_j}$ in \eqref{e1}.
These together in turn define a Hermitian structure $h$ on $E$ using the the isomorphism
in \eqref{e1} after imposing the condition that the subbundles
$E_j\otimes {\mathbb C}^{n_j}$ in \eqref{e1} are orthogonal.

The above Hermitian structure $h$ on $E$ clearly satisfies the Einstein--Hermitian
equation for the polystable vector bundle $E$.

Let $\theta^*\,\in\, C^\infty(M;\, End(E)\otimes\Omega^{0,1}_M)$ be the adjoint of
$\theta$. From the construction of $h$ it follows that $\theta\wedge\theta^*\,=\, 0$.
Using this it is straightforward to check that $h$ satisfies the Yang--Mills equation
for the Higgs bundle $(E\, ,\theta)$. 
\end{proof}

\begin{theorem}\label{thm1}
Let $(E\, ,\theta)$ be a polystable Higgs bundle on $M$. Let $h'$ be a
Yang--Mills Hermitian metric on $E$ for the Higgs field $\theta$. Then $h'$
satisfies the Einstein--Hermitian equation for the polystable vector bundle $E$.
\end{theorem}

\begin{proof}
Consider the Yang--Mills Hermitian metric $h$ on $E$ constructed in Proposition
\ref{pr1}. The two Hermitian structures $h$ and $h'$ differ by a holomorphic automorphism
of $E$. In other words, there is a holomorphic automorphism
$$
T\, :\, E\, \longrightarrow\, E
$$
such that
\begin{equation}\label{hp}
h'(v\, ,w)\,=\, h(T(v)\, ,T(w))
\end{equation}
for all $v\, , w\, \in\, E_x$ and all $x\, \in\, M$. From this we will derive that $h'$
satisfies the Einstein--Hermitian equation for the polystable vector bundle $E$.

Consider the holomorphic vector bundle $End(E)\,=\, E\otimes E^*$. The Hermitian
structure $h$ on $E$ produces a Hermitian structure on $End(E)$. The corresponding
Chern connection $\nabla$ on $End(E)$ is Einstein--Hermitian, because $h$ satisfies the
Einstein--Hermitian equation. Note that $c_1(End(E))\,=\, 0$. Therefore, the mean
curvature of the Einstein--Hermitian connection $\nabla$ on $End(E)$ vanishes identically
(see \cite[p. 51]{Ko} for mean curvature). Therefore, any holomorphic section of $End(E)$
is flat with respect to $\nabla$ \cite[p. 52, Theorem 1.9]{Ko}. In particular, the
the section $T$ in \eqref{hp} is flat with respect to $\nabla$.

Since $h$ satisfies the Einstein--Hermitian equation for $E$, and $T$ is flat with respect
to the connection $\nabla$ given by $h$, it follows that $h'$ defined by \eqref{hp} also
satisfies the Einstein--Hermitian equation for $E$.
\end{proof}

\section{Higgs $G$--bundles on $M$}

\subsection{Semistable and polystable Higgs $G$--bundles}\label{se4.1}

Let $G$ be a connected reductive affine algebraic group defined over $\mathbb C$. The
Lie algebra of $G$ will be denoted by $\mathfrak g$. For a holomorphic principal $G$--bundle
$E_G$ on $M$, let $\text{ad}(E_G)\,:=\, E_G\times^G \mathfrak g$ be the adjoint bundle.
A section
$$
\theta\,\in\, H^0(M,\, \text{ad}(E_G)\otimes\Omega^1_M)
$$
is called a \textit{Higgs field} on $E_G$ if $\theta\wedge\theta\,=\, 0$. A
{\it Higgs} $G$--{\it bundle} is a holomorphic principal $G$--bundle equipped with
a Higgs field.

The proof of the following lemma is very similar to the proof of Lemma \ref{lem1}.

\begin{lemma}\label{lem3}
Let $(E_G\, ,\theta)$ be a semistable Higgs $G$--bundle on $M$. Then the principal
$G$--bundle $E_G$ is semistable.
\end{lemma}

\begin{proof}
Assume that the principal $G$--bundle $E_G$ is not semistable. Let
$$
E_P\, \subset\, E_G
$$
be the Harder--Narasimhan reduction of $E_G$ over the dense open subset $U$ associated
to $E_G$. We have
$$
H^0(U,\, \text{ad}(E_G)/\text{ad}(E_P)) \,=\, 0
$$
\cite[p. 705, Corollary 1]{AAB}. Since the vector bundle $\Omega^1_M$ is trivial, this
implies that the image of $\theta$ in $H^0(U,\, (\text{ad}(E_G)/\text{ad}(E_P))\otimes
\Omega^1_M)$ vanishes identically. In other words,
$$
\theta\, \in\, H^0(U,\, \text{ad}(E_P)\otimes \Omega^1_M)\, .
$$
Therefore, the above reduction $E_P$ contradicts the given condition that
the Higgs $G$--bundle $(E_G\, ,\theta)$ is semistable. Consequently, the principal
$G$--bundle $E_G$ is semistable.
\end{proof}

Lemma \ref{lem3} and Lemma \ref{lem4} are proved in \cite{FGN2} under the
assumption that $M$ is an elliptic curve.

\begin{lemma}\label{lem4}
Let $(E_G\, ,\theta)$ be a polystable Higgs $G$--bundle on $M$. Then the principal
$G$--bundle $E_G$ is polystable.
\end{lemma}

\begin{proof}
Since $(E_G\, ,\theta)$ is polystable, it admits a Yang--Mills connection $\nabla$
\cite[p. 554, Theorem 4.6]{BS}. Let $\text{ad}(\theta)$ be the Higgs field on the
vector bundle $\text{ad}(E_G)$ induced by $\theta$. The connection on $\text{ad}(E_G)$
induced by $\nabla$ satisfies Yang--Mills equation for the Higgs bundle
$(\text{ad}(E_G)\, ,\text{ad}(\theta))$. Therefore,
$(\text{ad}(E_G)\, ,\text{ad}(\theta))$ is polystable. Hence the vector bundle
$\text{ad}(E_G)$ is polystable by Lemma \ref{lem2}. This implies that the principal
$G$--bundle $E_G$ is polystable \cite[p. 224, Corollary 3.8]{AB}.
\end{proof}

\subsection{A Levi reduction associated to a semisimple section}

Let $E_G$ be a holomorphic principal $G$--bundle over $M$. Let
\begin{equation}\label{eta}
\eta\, \in\, H^0(M,\, \text{ad}(E_G))
\end{equation}
be a section such that $\eta(x)\, \in\, \text{ad}(E_G)_x$ is semisimple for every $x\,
\in\, M$. Since the Lie algebra $\text{ad}(E_G)_x$ is identified
with the Lie algebra $\mathfrak g$ of $G$ up to an inner automorphism,
the element $\eta(x)\, \in\, \text{ad}(E_G)_x$ defines a conjugacy class in $\mathfrak g$.
Let
$$
C_x\, \subset\,\mathfrak g
$$
denote this orbit of $G$ in $\mathfrak g$ given by $\eta(x)$.

\begin{proposition}\label{prop2}
The above conjugacy class $C_x$ is independent of the point $x$.
\end{proposition}

\begin{proof}
Fix a maximal torus $T\, \subset\, G$. Let $W\, :=\, N(T)/T$ be the corresponding
Weyl group, where $N(T)\, \subset\, G$ is the normalizer of $T$. The Lie algebra
of $T$ will be denoted by $\mathfrak t$. The space of semisimple conjugacy classes in
$\mathfrak g$ is identified with the quotient ${\mathfrak t}/W$.

Since ${\mathfrak t}/W$ is an affine variety, and $M$ is a compact connected complex
manifold, there are no nonconstant holomorphic maps from $M$ to ${\mathfrak t}/W$. This
immediately implies that $C_x$ is independent of the point $x$.
\end{proof}

Fix an element
\begin{equation}\label{eta3}
\eta'\, \in\, C_x\, \subset\, {\mathfrak g}\, .
\end{equation}
Let
\begin{equation}\label{eta2}
L\,=\, C(\eta') \, \subset\, G
\end{equation}
be the centralizer of $\eta'$. It is known that $L$ is a Levi subgroup of $G$
\cite[p. 26, Proposition 1.22]{DM}; we recall that a Levi subgroup of $G$ is a maximal
connected reductive subgroup of some parabolic subgroup of $G$.

\begin{proposition}\label{prop3}
Given $\eta$ and $\eta'$ as above, the principal $G$--bundle $E_G$ has a natural
holomorphic reduction of structure group to the subgroup $L$ defined in \eqref{eta2}.
\end{proposition}

\begin{proof}
For any $g\, \in\, G$, let
$$
\text{Ad}(g)\, :\, {\mathfrak g}\, \longrightarrow\, {\mathfrak g}
$$
be the Lie algebra automorphism corresponding to the automorphism of the group $G$ defined
by $z\, \longmapsto\, g^{-1}zg$. We recall that $\text{ad}(E_G)$ is the quotient of
$E_G\times\mathfrak g$ where two point $(y_1\, ,v_1)\, ,(y_2\, ,v_2)\,\in\,
E_G\times\mathfrak g$ are identified if there is an element $g\, \in\, G$ such 
$y_2\,=\, y_1g$ and $v_2\,=\, \text{Ad}(g)(v_1)$. Let
$$
q\, :\, E_G\times{\mathfrak g}\,\longrightarrow\, \text{ad}(E_G)
$$
be the quotient map.

Let
$$
p_1\,:\, E_G\times{\mathfrak g}\,\longrightarrow\, E_G
$$
be the projection to the first factor. Define
\begin{equation}\label{cz}
{\mathcal Z}\,:=\, p_1((q^{-1}(\eta(M)))\cap (E_G\times \eta'))\, \subset\, E_G\, .
\end{equation}
It is straightforward to check that $\mathcal Z$ is a holomorphic reduction of structure
group of $E_G$ to the subgroup $L$.
\end{proof}

If $\eta'$ in \eqref{eta3} is replaced by $\text{Ad}(g)(\eta')$ for some $g\, \in\, G$,
then the subgroup $L$ in \eqref{eta2} gets replaced by $g^{-1}Lg$.

\begin{corollary}\label{cor2}
If $\eta'$ in \eqref{eta3} is replaced by ${\rm Ad}(g)(\eta')$ for some $g\, \in\, G$,
then $\mathcal Z$ in \eqref{cz} gets replaced by ${\mathcal Z}g\,\subset\,E_G$.
\end{corollary}

\begin{proof}
This follows immediately from the construction in \eqref{cz}.
\end{proof}

Let
$$
E_L\, \subset\, E_G
$$
be the reduction of structure group to $L$ constructed in Proposition \ref{prop3}.
Let $\text{ad}(E_L)$ be the adjoint vector bundle for $E_L$.

\begin{corollary}\label{cor3}
The subbundle ${\rm ad}(E_L)\, \subset\, {\rm ad}(E_G)$ is independent of the choice
of the element $\eta'$ in \eqref{eta3}.
\end{corollary}

\begin{proof}
This follows immediately from Corollary \ref{cor2}.
\end{proof}

\begin{corollary}\label{cor4}
For any $x\, \in\, M$, the subalgebra ${\rm ad}(E_L)_x\, \subset\, {\rm ad}(E_G)_x$
coincides with the centralizer of $\eta(x)\, \in\, {\rm ad}(E_G)_x$.
\end{corollary}

\begin{proof}
This follows from the construction of $E_L$ in \eqref{cz}.
\end{proof}

\subsection{A Levi reduction associated to a semisimple Higgs field}

Let $(E_G\, ,\theta)$ be a Higgs $G$--bundle on $M$. For any
$\alpha\, \in\, H^0(M,\, TM)$, we have
$$
\theta(\alpha)\, \in\, H^0(M,\, \text{ad}(E_G))\, .
$$
Take a basis $\{\alpha_1\, , \cdots\, , \alpha_d\}$ of $H^0(M,\, TM)$.

\begin{lemma}\label{lem5}
If $\theta(\alpha_i)$ is pointwise semisimple for every $i\, \in\, \{1\, ,\cdots\, ,
d\}$, then for any $\alpha\, \in\, H^0(M,\, TM)$, the section $\theta(\alpha)$ is
pointwise semisimple.
\end{lemma}

\begin{proof}
Since $\theta$ is a Higgs field, we have $[\theta(\alpha_i)\, ,\theta(\alpha_j)]\,=\,
0$ for every $i\, ,j\,\in\, \{1\, ,\cdots\, , d\}$. The lemma follows from the
fact that a sum of commuting semisimple elements of $\mathfrak g$ is again semisimple.
\end{proof}

Assume that $\theta(\alpha_i)$ is pointwise semisimple for every
$i\, \in\, \{1\, ,\cdots\, , d\}$. Let $L_1\, \subset\, G$ be the Levi subgroup
constructed as in \eqref{eta2} for the section $\theta(\alpha_1)$. Let
$$
E_{L_1}\, \subset\, E_G
$$
be the reduction constructed as in Proposition \ref{prop2} for $\theta(\alpha_1)$.
Since $[\theta(\alpha_1)\, ,\theta(\alpha_2)]\,=\, 0$, from Corollary \ref{cor4}
we know that
$$
\theta(\alpha_2)\, \subset\, H^0(M,\, \text{ad}(E_{L_1}))\,\subset\,
H^0(M,\, \text{ad}(E_G))\, .
$$
Therefore, proceeding inductively, we get from the Higgs field $\theta$
\begin{itemize}
\item a Levi subgroup $L\, \subset\, G$, and

\item a holomorphic reduction of structure group
\begin{equation}\label{el}
E_L\, \subset\, E_G
\end{equation}
to $L$.
\end{itemize}

The subgroup $L$ is unique up to a conjugation. If $L$ is replaced by
$g^{-1}Lg$ for some $g\, \in\, G$, then $E_L$ gets replaced by $E_Lg$. Consequently,
the subbundle
$$
\text{ad}(E_L)\, \subset\, \text{ad}(E_G)
$$
is uniquely determined by $\theta$. Also, note that
\begin{equation}\label{el2}
\theta\, \in\, H^0(M,\, \text{ad}(E_L)\otimes\Omega^1_M) \,\subset\,
H^0(M,\, \text{ad}(E_G)\otimes\Omega^1_M)\, .
\end{equation}
{}From Corollary \ref{cor4} it follows that for every $i\, \in\, \{1\, ,\cdots\, , d\}$
and $x\, \in\, M$, the subalgebra $\text{ad}(E_L)_x\, \subset\, \text{ad}(E_G)_x$ is
contained in the centralizer of $\theta(\alpha_i)(x)$. More precisely,
$\text{ad}(E_L)_x$ is the centralizer of the subset $\{\theta(\alpha_1)(x)\, ,
\cdots\, , \theta(\alpha_d)(x)\}\,\subset\, \text{ad}(E_G)_x$.

\section{Yang--Mills structure on polystable Higgs $G$--bundles on $M$}

Let $(E_G\, ,\theta)$ be a polystable Higgs $G$--bundle on $M$.

\begin{proposition}\label{prop4}
For any $i\, \in\, \{1\, ,\cdots\, , d\}$ and $x\, \in\, M$, the element
$\theta(\alpha_i)(x)\,\in\, {\rm ad}(E_G)_x$ is semisimple.
\end{proposition}

\begin{proof}
Let
\begin{equation}\label{zg}
Z(G)\, \subset\, G
\end{equation}
be the connected component of the center of $G$
containing the identity element. Take a finite dimensional holomorphic representation
$$
\rho\, :\,G\,\longrightarrow\, \text{GL}(V)
$$
such that $\rho(Z(G))$ is contained in the center of $\text{GL}(V)$. Let
$$
E_V\, :=\, E_G\times^G V\, \longrightarrow\, M
$$
be the vector bundle associated to $E_G$ for this $G$--module $V$. The Higgs field
$\theta$ induces a Higgs field on $E_V$. This induced Higgs field on $E_V$ will be
denoted by $\theta_V$. The connection on $E_V$ induced by a Yang--Mills connection
for $(E_G\, ,\theta)$ satisfies the Yang--Mills equation for the Higgs bundle
$(E_V\, ,\theta_V)$. This implies that $(E_V\, ,\theta_V)$ is polystable.

Now from Proposition \ref{prop1} we conclude that $\theta_V(\alpha_i)(x)
\,\in\, \text{End}(E_V)_x$ is semisimple for every
$i\, \in\, \{1\, ,\cdots\, , d\}$ and $x\, \in\, M$. Since $\rho$ is an arbitrary
holomorphic representation such that $\rho(Z(G))$ is contained in the center of
$\text{GL}(V)$, this implies that $\theta(\alpha_i)(x)$ is semisimple.
\end{proof}

Consider the Higgs $L$--bundle $(E_L\, ,\theta)$ constructed from the given polystable
Higgs $G$--bundle $(E_G\, ,\theta)$ (see \eqref{el}, \eqref{el2}). We note that
$(E_L\, ,\theta)$ is polystable because $(E_G\, ,\theta)$ is so. From Lemma \ref{lem4}
we know that the principal $L$--bundle $E_L$ is polystable.

As in Section \ref{sec3}, fix the K\"ahler form $\widetilde\omega$ on $M$. 
Fix a maximal compact subgroup
$$
K_L\, \subset\, L\, .
$$
Let
$$
E_{K_L}\, \subset\, E_L
$$
be a $C^\infty$ reduction of structure group to $K_L$ that solves the Yang--Mills
equation for $(E_L\, ,\theta)$ (see \cite[p. 554, Theorem 4.6]{BS}).

\begin{proposition}\label{pr2}
The above reduction
$$
E_{K_L}\, \subset\, E_L
$$
solves the Einstein--Hermitian equation for the polystable principal $L$--bundle
$E_L$.
\end{proposition}

\begin{proof}
We observed earlier that for every $i\, \in\, \{1\, ,\cdots\, , d\}$
and $x\, \in\, M$, the subalgebra $\text{ad}(E_L)_x\, \subset\, \text{ad}(E_G)_x$ is
contained in the centralizer of $\theta(\alpha_i)(x)$. Therefore,
for every $i\, \in\, \{1\, ,\cdots\, , d\}$
and $x\, \in\, M$, the element $\theta^*(\overline{\alpha_i})(x)
\,\in\, \text{ad}(E_L)_x$ also is
contained in the center of $\text{ad}(E_L)_x$. Consequently, we have
$\theta\wedge\theta^*\,=\, 0$. This immediately implies that the reduction
$$
E_{K_L}\, \subset\, E_L
$$
solves the Einstein--Hermitian equation for the polystable principal $L$--bundle
$E_L$.
\end{proof}

Fix a maximal compact subgroup
$$
K\, \subset\, G
$$
such that $K\cap L\,=\, K_L$.

\begin{theorem}\label{thm2}
Let $(E_G\, ,\theta)$ be a polystable Higgs $G$--bundle on $M$.
Let
$$
E_K\, \subset\, E_G
$$
be a $C^\infty$ reduction of structure group to $K$ that solves the Yang--Mills equation
for $(E_G\, ,\theta)$. Then the reduction $E_K\, \subset\, E_G$ solves the
Einstein--Hermitian equation for the polystable principal $G$--bundle $E_G$.
\end{theorem}

\begin{proof}
As before, let $(E_L\, ,\theta)$ be the Higgs $L$--bundle constructed from the
polystable Higgs $G$--bundle $(E_G\, ,\theta)$ (see \eqref{el}, \eqref{el2}). Take
a $C^\infty$ reduction
$$
E_{K_L}\, \subset\, E_L
$$
that solves the Yang--Mills equation for $(E_L\, ,\theta)$. Let
$$
E'_K\, :=\, E_{K_L}(K)\,\longrightarrow\, M
$$
be the principal $K$--bundle obtained by extending the structure group of
$E_{K_L}$ using the inclusion of $K_L$ in $K$. We note that $E'_K$ is a reduction
of structure group of $E_G$ to $K$ because $E_{K_L}$ is a reduction
of structure group of $E_G$ to $K_L$. The above reduction
$$
E'_K\, \subset\, E_G
$$
solves the Yang--Mills equation for $(E_G\, ,\theta)$ because the reduction
$E_{K_L}\,\subset\, E_L$ solves the Yang--Mills equation for $(E_L\, ,\theta)$.

Therefore, there is a holomorphic automorphism $T$ of $E_G$ such that
$E_K\,=\, T(E'_K)$.

Let $\text{Ad}(E_G)\,=\, E_G\times^G G\,\longrightarrow \, M$ be the
holomorphic fiber bundle associated to
$E_G$ for the adjoint action of $G$ on itself. It can be shown that the holomorphic
sections of $\text{Ad}(E_G)$ are flat with respect to the connection on $\text{Ad}(E_G)$
induced by the connection on $E_G$ given by the reduction $E'_K$. To prove this, take any
finite dimensional holomorphic $G$--module
$$
\rho\, :\, G\, \longrightarrow\, \text{GL}(V)
$$
such that $\rho(Z(G))$ (see \eqref{zg}) is contained in the center of $\text{GL}(V)$. Let 
$$
E_V\,:=\, E_G\times^G V\,\longrightarrow \, M
$$
be the associated vector bundle. The Einstein--Hermitian connection on $E_G$ given by
the reduction $E'_K$ produces an Einstein--Hermitian connection on ${\rm End}(E_V)$;
this Einstein--Hermitian connection on ${\rm End}(E_V)$ will be denoted by
$\nabla'$.

Given any holomorphic section $T'$ of $\text{Ad}(E_G)$, let
$T''$ be the automorphism of $E_V$ given by $T'$. As done in the proof of
Theorem \ref{thm1}, using \cite[p. 52, Theorem 1.9]{Ko} we conclude that the section
$T''$ of $End(E)$ is flat with respect to $\nabla'$. From this it follows that the
automorphism $T'$ of $E_G$ is flat with respect to the connection on $\text{Ad}(E_G)$
induced by the connection on $E_G$ given by the reduction $E'_K$.

In particular, the earlier automorphism $T$ is flat
with respect to the connection on $\text{Ad}(E_G)$
corresponding to the reduction $E'_K$. From this it follows that the 
reduction $E_K\, \subset\, E_G$ solves the
Einstein--Hermitian equation for the polystable principal $G$--bundle $E_G$.
\end{proof}

\section{Higgs bundles on K\"ahler manifolds with nonnegative tangent bundle}

Let $X$ be a compact connected K\"ahler manifold equipped with a K\"ahler class
$\omega$. Let
$$
W_1\, \subset\, \cdots \, \subset\, W_m\,=\, \Omega^1_X
$$
be the Harder--Narasimhan filtration of $\Omega^1_X$.

\begin{lemma}\label{le1}
Assume that $\mu_{\rm max}(\Omega^1_X)\,:=\, \mu(W_1) \, <\, 0$. Let $(E
\, , \theta)$ be a semistable Higgs bundle on $X$. Then $\theta\,=\, 0$.
\end{lemma}

\begin{proof}
To prove that $E$ is semistable, consider $E_1$ in \eqref{ee1}.
We have
$$
H^0(X,\, Hom(E_1\, ,(E/E_1)\otimes \Omega^1_X))\,=\, 0
$$
because $\mu_{\rm max}((E/E_1)\otimes \Omega^1_X)\,=\, \mu_{\rm max}(E/E_1)+
\mu_{\rm max}(\Omega^1_X)\, <\, \mu(E_1)+0 \,=\,\mu(E_1)$. Form this it follows
that $\theta(E_1)\, \subset\, E_1\otimes\Omega^1_X$. Since
$(E \, , \theta)$ is semistable, this implies that $E\,=\,E_1$. So $E$ is semistable.

Since $E$ is semistable,
$$
\mu_{\rm max}(E\otimes \Omega^1_X)\,=\,
\mu(E)+ \mu_{\rm max}(\Omega^1_X)\, <\, \mu(E)\, .
$$
Hence $H^0(X,\, End(E)\otimes \Omega^1_X)\,=\, 0$. In particular, $\theta\,=\, 0$.
\end{proof}

Combining the proofs of Lemma \ref{lem3} and Lemma \ref{le1} it is easy to deduce
that Lemma \ref{le1} remains valid for Higgs $G$--bundles on $X$. The only point to
note is that \cite[p. 705, Corollary 1]{AAB} (which is used in the proof of
Lemma \ref{lem3}) is proved by showing that
$\mu_{\rm max}(\text{ad}(E_G)/\text{ad}(E_P))\, <\, 0$.

\subsection{Higgs bundles on Calabi--Yau manifolds}

Let $X$ be a compact connected K\"ahler manifold such that $c_1(TX) \,\in\, H^2(X,\,
{\mathbb Q})$ is zero. These are known as Calabi--Yau manifolds. Fix a K\"ahler class
$\omega$ on $X$. A celebrated theorem of Yau says that there is a K\"ahler form
$\widetilde\omega$ in the class $\omega$ such that the Ricci curvature for $\widetilde
\omega$ vanishes identically \cite{Ya} (this was conjectured earlier by Calabi). In
particular, $\widetilde\omega$ is an Einstein--Hermitian structure on $\Omega^1_X$. This
implies that the vector bundle $\Omega^1_X$ is polystable.

\begin{lemma}\label{le2}
Let $(E \, , \theta)$ be a semistable Higgs bundle on $X$. Then the vector bundle
$E$ is semistable. This is also true for Higgs $G$--bundles, meaning if $(E_G \, ,
\theta)$ is a semistable Higgs $G$--bundle on $X$, then the underlying principal
$G$--bundle $E_G$ is semistable.
\end{lemma}

\begin{proof}
Since $\Omega^1_X$ is polystable of slope zero, the proof of it given in Lemma
\ref{le1} remains valid. To prove for Higgs $G$--bundles, just note that
for $\text{ad}(E_P)$ in the proof of Lemma \ref{lem3} we have
$\mu_{\rm max}(\text{ad}(E_G)/\text{ad}(E_P))\, <\, 0$ (see the proof of
Corollary 1 in \cite[p. 705]{AAB}).
\end{proof}

\begin{lemma}\label{le3}
Let $(E \, , \theta)$ be a polystable Higgs bundle on $X$. Then the vector bundle
$E$ is polystable.
\end{lemma}

\begin{proof}
{}From Lemma \ref{le2} we know that $E$ is semistable. As in the proof of Lemma
\ref{lem2},
$$
V\,\subset\, E
$$
is the coherent analytic subsheaf generated by all polystable subsheaves of $E$
with slope $\mu(E)$. Let $\theta'\,:\, TX\otimes E\, \longrightarrow\, E$
be the following composition homomorphism
$$
TX\otimes E\,\stackrel{\text{Id}_{TX}\otimes\theta}{\longrightarrow}\, TX\otimes
\Omega^1_X\otimes E\,\stackrel{{\rm trace}\otimes{\rm Id}_E}{\longrightarrow}\, E\, .
$$
Since both $TX$ and $V$ are polystable, it follows that $TX\otimes V$ is polystable.
Also, note that $\mu(TX\otimes V)\,=\, \mu(V)\,=\, \mu(E)$. Therefore, the image
$$
\theta'(TX\otimes V)\, \subset\, E
$$
is polystable with $\mu(\theta'(TX\otimes V))\,=\, \mu(E)$. Hence, we have
$$
\theta'(TX\otimes V)\, \subset\, V\, .
$$
This implies that $\theta(V)\, \subset\, V\otimes\Omega^1_X$. Now the last part
of the proof of Lemma \ref{lem2} shows that $E$ is polystable.
\end{proof}

\begin{lemma}\label{le4}
Let $(E_G \, , \theta)$ be a polystable Higgs $G$--bundle on $X$. Then the principal
$G$-- bundle $E_G$ is polystable.
\end{lemma}

\begin{proof}
The proof is identical to the proof of Lemma \ref{lem4}.
\end{proof}


\end{document}